\documentclass[11pt, a4paper, english]{smfart}

\usepackage{smfthm}

\usepackage{amssymb}
\usepackage{tikz-cd}
\usepackage{mathtools}

\usepackage{hyperref}
\usepackage{enumitem}
\usepackage{comment}
\usepackage{cleveref}

\renewcommand{\geq}{\geqslant}
\renewcommand{\leq}{\leqslant}

\theoremstyle{plain}
\newtheorem{theorem}{Theorem}[section]

\newtheorem{lemma}[theorem]{Lemma}
\newtheorem{conjecture}[theorem]{Conjecture}
\newtheorem{corollary}[theorem]{Corollary}
\newtheorem{proposition}[theorem]{Proposition}

\theoremstyle{remark}
\newtheorem{definition}[theorem]{Definition}
\newtheorem{remark}[theorem]{Remark}
\newtheorem{example}[theorem]{Example}

\renewcommand{\AA}{\ensuremath{\mathbb{A}}}
\newcommand{\CC}{\ensuremath{\mathbb{C}}}

\newcommand{\PP}{\ensuremath{\mathbb{P}}}
\newcommand{\QQ}{\mathbb{Q}}

\newcommand{\ZZ}{\ensuremath{\mathbb{Z}}}

\DeclareMathOperator{\GL}{GL}

\DeclareMathOperator{\Spec}{Spec}
\DeclareMathOperator{\End}{End}

\DeclareMathOperator{\Gal}{Gal}

\DeclareMathOperator{\MT}{MT}

\DeclareMathOperator{\h}{\mathsf{h}}

\title[Weil fourfolds with discriminant 1 and singular OG6-varieties]{The Hodge conjecture for Weil fourfolds with discriminant 1 via singular OG6-varieties}

\author{Salvatore Floccari}
\address{Humboldt-Universität zu Berlin, Institut für Mathematik, Germany}
\email{salvatore.floccari@hu-berlin.de}

\author{Lie Fu}
\address{Institut de Recherche Mathématique Avancée, Universit\'e de Strasbourg, France}
\email{lie.fu@math.unistra.fr}

\begin{document}


	\begin{altabstract}
		Nous donnons une nouvelle démonstration de la conjecture de Hodge pour les variétés abéliennes de dimension 4 de type Weil de discriminant 1, ainsi que pour toutes leurs puissances. La conjecture de Hodge pour ces variétés abéliennes de dimension 4 a été démontrée par Markman à l’aide de faisceaux hyperholomorphes sur des variétés hyperkählériennes de type Kummer généralisé, ainsi que par la construction de faisceaux semiréguliers sur des variétés abéliennes. Notre démonstration repose au contraire sur une relation géométrique directe entre les variétés abéliennes de dimension 4 de type Weil de discriminant 1, et les variétés hyperkählériennes $\widetilde{K}$ de dimension six de type O’Grady, qui apparaissent comme des résolutions symplectiques $\widetilde{K} \to K$ d’une déformation localement triviale d’un espace de modules singulier de faisceaux sur une surface abélienne. Comme applications, nous établissons la conjecture de Hodge et la conjecture de Tate pour toute variété $\widetilde{K}$ de type OG6 comme ci-dessus, ainsi que pour toutes ses puissances.
	\end{altabstract}
	
	\begin{abstract}
		We give a new proof of the Hodge conjecture for abelian fourfolds of Weil type with discriminant 1 and all of their powers.
		The Hodge conjecture for these abelian fourfolds was proven by Markman using hyperholomorphic sheaves on hyper-K\"ahler varieties of generalized Kummer type, and by constructing semiregular sheaves on abelian varieties. 
		Our proof instead relies on a direct geometric relation between abelian fourfolds of Weil type with discriminant $1$ and the six-dimensional hyper-K\"ahler varieties $\widetilde{K}$ of O'Grady type arising as crepant resolutions $\widetilde{K}\to K$ of a locally trivial deformation of a singular moduli space of sheaves on an abelian surface. As applications, we establish the Hodge conjecture and the Tate conjecture for any variety $\widetilde{K}$ of OG6-type as above, and all of its powers.
	\end{abstract}

	\maketitle

	\section{Introduction}
	
	The Hodge conjecture predicts profound and fundamental properties of complex projective manifolds.
	Abelian varieties of Weil type \cite{Weil} constitute an important testing ground for the Hodge conjecture, since the middle cohomology of such an abelian variety contains certain Hodge classes, called \textit{Hodge--Weil classes}, which, in the generic case, cannot be expressed as linear combination of intersections of divisor classes.  
	
	Let us recall the definition of these abelian varieties. For a positive integer $d$, an abelian variety $A$ of even dimension $2n$ is of $\QQ(\sqrt{-d})$-\textit{Weil type} if $\mathrm{End}_{\QQ}(A)$ contains the imaginary quadratic field $\QQ(\sqrt{-d})$ and the action of $\sqrt{-d}$ on the tangent space of the origin of $A$ has eigenvalues $\sqrt{-d}$ and $-\sqrt{-d}$,  both with multiplicity $n$. In each even dimension $2n$, abelian varieties of Weil type form countably many families of dimension~$n^2$, where each family is characterized by the field $\QQ(\sqrt{-d})$ and another discrete invariant $\delta\in \QQ^{*}/\mathrm{Nm}(\QQ(\sqrt{-d})^*)$, which is called the discriminant.
	
	Until recently, even in dimension $4$, only sporadic cases were known about the algebraicity of the Hodge--Weil classes: it was proven by Schoen \cite{Schoen} for abelian fourfolds of $\QQ(\sqrt{-3})$-Weil type with arbitrary discriminant or of  $\QQ(\sqrt{-1})$-Weil type with discriminant 1 (see also van Geemen \cite{vanGeemen1996} for the latter case).  
	
	A breakthrough in this direction is the following theorem of Markman \cite{markman2019monodromy}.
	
	\begin{theorem}[Markman] \label{thm:1}
		The Hodge--Weil classes are algebraic for any abelian fourfold of Weil type with discriminant $1$.
	\end{theorem}
	
	Markman's proof of the above theorem in \cite{markman2019monodromy} relies on the tight link between such abelian fourfolds and hyper-K\"ahler varieties of generalized Kummer type, discovered by himself and O'Grady~\cite{O'G21}. 
	Via results of Voisin \cite{voisinfootnotes}, Varesco \cite{varesco} and the first author~\cite{floccariKum3}, the above theorem has strong consequences for the Hodge conjecture for hyper-K\"ahler varieties of generalized Kummer type, see \cite{floccariHCKum3, floccariVaresco}.
	
	Building on the aforementioned works \cite{markman2019monodromy, voisinfootnotes, varesco},
	Theorem \ref{thm:1} was strengthened by the first author in \cite{floccari25}, as follows.
	\begin{theorem}\label{thm:1+}
		The Hodge conjecture holds for all powers of any abelian fourfold of Weil type with discriminant $1$.
	\end{theorem}
	
	Markman has recently given in \cite{Markman-Weil6folds} another proof of Theorem \ref{thm:1}. He is able to obtain the much stronger result that the Hodge--Weil classes are algebraic for \textit{any} abelian fourfold of Weil type, with no restriction on the discriminant. Most remarkably, combined with the work of Moonen and Zahrin \cite{moonenZahrinHodge, moonenZarhin}, this result finally settles the Hodge conjecture for all abelian varieties of dimension $4$.
	
	The purpose of this paper is to provide an alternative proof of Theorem \ref{thm:1} and Theorem \ref{thm:1+},  which makes these results, as well as the subsequent applications, independent from the rather involved construction of hyperholomorphic sheaves in \cite{markman2019monodromy} (or of semi-regular sheaves in \cite{Markman-Weil6folds}).

	\subsection*{Ingredients of the proof}
	Our approach towards Theorem \ref{thm:1} is inspired by the beautiful articles of Markman~\cite{markman2019monodromy}, O'Grady~\cite{O'G21} and Voisin \cite{voisinfootnotes} on the Kuga--Satake construction for varieties of generalized Kummer type, but we work instead with the six-dimensional hyper-K\"ahler varieties discovered by O'Grady~\cite{O'G03}. In fact, we only consider the divisors in their moduli spaces given by the crepant resolutions of \textit{singular OG6-varieties}, which we define as the locally trivial deformations of the Albanese fibre of the singular moduli space of sheaves studied in~\cite{O'G03}; see Definition \ref{def:SingularOG6}. 
	These singular varieties carry pure Hodge structures of K3-type on their second cohomology. 
	We study the Kuga--Satake construction for singular OG6-varieties. 
	
	\begin{theorem}\label{thm:KSsingularOG6}
		Let $K$ be a projective singular $\mathrm{OG}6$-variety. Then:
		\begin{enumerate}[label=(\roman*)]
			\item the singular locus of $K$ is isomorphic to $B_K/\pm 1$ for a $4$-dimensional abelian variety $B_K$, which is of Weil type with discriminant $1$;
			\item the Kuga--Satake variety of $K$ is isogenous to $B^4_K$.
		\end{enumerate}
		Moreover, any abelian fourfold $B$ of Weil type with discriminant $1$ is isogenous to $B_K$ for some projective singular $\mathrm{OG}6$-variety $K$. 
	\end{theorem}

	Any singular OG6-variety $K$ is also naturally associated with a K3 surface $S_K$. Indeed, by the results of Mongardi--Rapagnetta--Sacc\`a \cite{MRS18}, any such $K$ admits a rational double cover $Z_K \dashrightarrow K$ where $Z_K$ is a variety of $\mathrm{K}3^{[3]}$-type birational to a moduli space of sheaves on a projective K3 surface $S_K$, uniquely determined up to isomorphism. These K3 surfaces come in $4$-dimensional families of generic Picard rank $16$. 
	
	The Kuga--Satake variety of $S_K$ is isogenous to a power of that of~$K$, and, hence, to a power of $B_K$. By a theorem of Varesco \cite{varesco} and Theorem \ref{thm:KSsingularOG6}, we can reduce Theorem \ref{thm:1} to verifying the so-called Kuga--Satake--Hodge conjecture in this case, namely, showing that the Kuga--Satake correspondence is algebraic for all the K3 surfaces $S_K$. To prove this, we exploit that the double cover $Z_K\dashrightarrow K$ ramifies over the singular locus of $K$, and the ramification locus in~$Z_K$ is birational to $B_K/\pm 1$. The rational embedding of $B_K/\pm 1$ into $Z_K$ yields a cycle which we use to prove that the Kuga--Satake correspondence is algebraic for $S_K$. 
	The proof explains how to recover the Kuga--Satake abelian fourfold $B$ from the K3 surface~$S$: there exists a moduli space of sheaves on $S$ with a birational involution whose ``fixed locus'' is birational to~$B/\pm 1$.
	
	Once the Kuga--Satake--Hodge conjecture for these K3 surfaces is established, we deduce Theorem \ref{thm:1+} using the strategy developed in \cite{floccari25}, hence obtaining a new proof of Theorem \ref{thm:1+} which does not rely on Markman's results in \cite{markman2019monodromy} and \cite{Markman-Weil6folds}.
	
	In the final section, we give applications to the Hodge and Tate conjectures and the study of the homological motive of hyper-K\"ahler varieties. We refer to Section \ref{sec:applications} for the most general statements (Theorem \ref{thm:motivated}, Corollary \ref{cor:applicationHC}) and only mention the following results. 
	
	\begin{theorem}\label{thm:HCOG6-resolutions}
		Let $\widetilde{K}$ be a hyper-K\"ahler variety of $\mathrm{OG}6$-type obtained as a crepant resolution of a singular $\mathrm{OG}6$-variety $K$. Then: 
		\begin{enumerate}[label=(\roman*)]
			\item the Hodge conjecture holds for $\widetilde{K}$ and all of its powers. 
			\item the Kuga--Satake--Hodge conjecture holds for $\widetilde{K}$.
			\item the homological motive of $\widetilde{K}$ is \emph{abelian}, that is, it belongs to the $\QQ$-linear tensor subcategory generated by the motives of abelian varieties. 
		\end{enumerate}
	\end{theorem}
	
	Let now $L\subset \CC$ be a subfield which is finitely generated over $\QQ$. For a smooth projective variety $X$ defined over $L$ and a prime number $\ell$, consider the representation of the absolute Galois group of $L$ on the $\ell$-adic cohomology groups $H^i_{\text{\'et}}(X_{\bar{L}},\QQ_{\ell})$. The strong Tate conjecture predicts that these Galois representations are semisimple and that the cycle class map $\mathrm{cl}^i\colon \mathrm{CH}^i(X)_{\QQ_{\ell}}\to H^{2i}_{\text{\'et}}(X,\QQ_{\ell}(i))^{\Gal(\bar{L}/L)}$ is surjective for any integer $i$.
	
	\begin{theorem}\label{thm:TateConjectureOG6resolutions}
		Let $X$ be a smooth projective variety over $L$ whose base-change $X_{\CC}$ is a crepant resolution of a singular $\mathrm{OG}6$-variety.
		Then, for any prime number $\ell$, the strong Tate conjecture holds for $X$ and all of its powers.
	\end{theorem}
	
	We deduce this result from Theorem \ref{thm:HCOG6-resolutions} using the Mumford--Tate conjecture, which has been proven for all known hyper-K\"ahler varieties \cite{floccari2019, soldatenkov19, FFZ}. See Corollary \ref{cor:TateConjecture} for a more general statement.

	\subsection*{Acknowledgments}
	We are grateful to Kieran O'Grady, Eyal Markman and Claire Voisin for their interest in this project. We thank the anonymous referee for his or her careful reading and many useful comments.
	Salvatore Floccari is supported by the Deutsche Forschungsgemeinschaft (DFG, German Research Foundation) – Project-ID 491392403 – TRR 358. Lie Fu is supported by the University of Strasbourg Institute for Advanced Study (USIAS), by the Agence Nationale de la Recherche (ANR) under projects ANR-20-CE40-0023 and ANR-24-CE40-4098, and by the International Emerging Actions (IEA) of CNRS.

	\section{Preliminaries on the Kuga--Satake construction}
	
	A rational Hodge structure of \textit{K3-type} is an effective pure $\QQ$-Hodge structure $V$ of weight~$2$ such that $V^{2,0}$ is $1$-dimensional. Starting with a polarized rational Hodge structure of K3-type $(V,q)$, where $q$ denotes the polarization on the Hodge structure $V$, the Kuga--Satake construction \cite{KUGA1967, deligne1971conjecture} produces an abelian variety $\mathrm{KS}(V)$, well-defined up to isogeny\footnote{If one starts with a polarized integral Hodge structure of K3-type, the Kuga--Satake construction gives rise to a well-defined abelian variety. We will only need the construction up to isogeny.}. The crucial property of $\mathrm{KS}(V)$ is the existence of an embedding of rational Hodge structures 
	\[
	\mu\colon V \hookrightarrow H^2(\mathrm{KS}(V)\times \mathrm{KS}(V),\QQ).
	\]
	
	Let us briefly recall the construction, following \cite{deligne1971conjecture}. Let $(V,q)$ be as above. Let $\mathsf{T}^{\bullet}V=\bigoplus_{i\geq 0} V^{\otimes i}$ be the tensor algebra of $V$. One considers the Clifford algebra $$C(V)\coloneqq \mathsf{T}^{\bullet}V / \langle v\otimes v - q(v,v)\rangle_{v\in V},$$ which is naturally $\ZZ/2\ZZ$-graded: $C(V)=C^+(V)\oplus C^-(V)$, where $C^+(V)$ (resp. $C^{-}(V)$) is the subspace generated by products of an even (resp. odd) number of vectors in $V$; notice that $C^+(V)$ is a subalgebra of $C(V)$, called the even Clifford algebra, and that $C^-(V)$ is a $C^+(V)$-bimodule via left and right multiplication.
	The Hodge structure on $V$ induces an effective Hodge structure of weight $1$ on $C^+(V)$, and $C^+(V)$ with this Hodge structure is isomorphic to $H^1(\mathrm{KS}(V),\QQ)$ for an abelian variety $\mathrm{KS}(V)$, called the Kuga--Satake variety of $(V,q)$. Fix a vector $v_0\in V$ such that $q(v_0,v_0)\neq 0$. The map $\mu\colon V\to \End(C^+(V))$ given by $\mu(v)(w)=v\cdot w\cdot v_0$ is an embedding of weight-$0$ Hodge structures 
	\[
	{\mu}\colon V(1) \hookrightarrow \End(H^1(\mathrm{KS}(V),\QQ)).
	\]
	Identifying $\End(H^1(\mathrm{KS}(V),\QQ))$ with $H^1(\mathrm{KS}(V),\QQ)^{\otimes 2}(1)$ via a polarization, we obtain an embedding of weight-$2$ Hodge structures 
	\[
	\mu\colon V\hookrightarrow H^2(\mathrm{KS}(V)^2,\QQ).
	\]
	This embedding is called the Kuga--Satake embedding. 
	
	\begin{remark}\label{rmk:functorialityKS}
		Let $(V_1,q_1)$ and $(V_2,q_2)$ be polarized rational Hodge structures of K3-type, and assume that $f\colon V_1\to V_2$ is a Hodge similarity, i.e., an isomorphism of $\QQ$-Hodge structures which multiplies the form by a factor $k\in \QQ^*$. Then, the Kuga--Satake varieties $\mathrm{KS}(V_1)$ and $\mathrm{KS}(V_2)$ are isogenous. In fact, by \cite[Proposition 0.2]{varesco2023hodge}, there exists an isogeny $F\colon \mathrm{KS}(V_1)\to \mathrm{KS}(V_2)$ and a commutative diagram
		\[
		\begin{tikzcd}
			V_1 \arrow{d}{\mu_1} \arrow{r}{f} & V_2 \arrow{d}{\mu_2}\\
			H^2(\mathrm{KS}(V_1)^2, \QQ) \arrow{r}{F_*} & H^2(\mathrm{KS}(V_2)^2,\QQ),
		\end{tikzcd}
		\]
		where $F_*$ is the isomorphism of Hodge structures induced by $F$,
		and $\mu_1$, $\mu_2$ are the respective Kuga--Satake embeddings.
	\end{remark}
	
	Let $(X,h)$ be a polarized hyper-K\"ahler variety, with $H^2(X, \QQ)$ equipped with the Beauville--Bogomolov quadratic form \cite{beauville1983varietes}. Let $H^2(X,\QQ)_{\mathrm{prim}}$ be the second primitive cohomology of $X$, which is identified with the orthogonal complement of the polarization class $h$ with respect to the Beauville--Bogomolov quadratic form by the Fujiki relation \cite{Fujiki-1985}. 
	Then $H^2(X,\QQ)_{\mathrm{prim}}$ equipped with the restriction of the quadratic form is a polarized Hodge structure of K3-type. 
	
	\begin{definition} \label{def:KS-variety} 
		The Kuga--Satake variety $\mathrm{KS}(X)$ of the polarized hyper-K\"ahler variety $(X,h)$ is the abelian variety $\mathrm{KS}(H^2(X,\QQ)_{\mathrm{prim}})$ obtained via the Kuga--Satake construction, well-defined up to isogeny.
	\end{definition} 
	
	\begin{remark}\label{rmk:quasi-polarized}
		Let $(X,h)$ be a quasi-polarized hyper-K\"ahler variety, where $h$ is the cohomology class of a big and nef line bundle. Then $h\in H^2(X,\QQ)$ is a positive class, and its orthogonal complement $H^2(X,\QQ)_{\mathrm{prim}}$, equipped with the restriction of the Beauville--Bogomolov form, is a polarized Hodge structure of K3-type. We then define the Kuga--Satake variety $\mathrm{KS}(X)$ of $(X,h)$ as the abelian variety obtained from $H^2(X,\QQ)_{\mathrm{prim}}$ via the Kuga--Satake construction; everything that we will say below applies to this more general setting.
	\end{remark}
	
	The Hodge conjecture predicts that the embedding $\mu$ should be the correspondence induced by an algebraic cycle on $X\times \mathrm{KS}(X)\times \mathrm{KS}(X)$. This leads to the following conjecture, which is of central importance in this paper.
	
	\begin{conjecture}[Kuga--Satake--Hodge conjecture]\label{conj:KSHC}
		Let $(X,h)$ be a $2n$-dimensional polarized hyper-K\"ahler variety. Then, there exists an algebraic class $\gamma\in H^{4n-2}(X,\QQ)\otimes H^2(\mathrm{KS}(X)^2,\QQ)\subset H^{4n}(X\times \mathrm{KS}(X)^2, \QQ)$ which induces, as a correspondence, an embedding of rational Hodge structures $$H^2(X,\QQ)_{\mathrm{prim}}\hookrightarrow H^2(\mathrm{KS}(X)^2,\QQ).$$
	\end{conjecture} 
	
	\begin{remark}\label{rmk:KStr}
		The transcendental cohomology $H^2_{\mathrm{tr}}(X,\QQ)$ of a hyper-K\"ahler variety~$X$ is the orthogonal complement in $H^2(X, \QQ)$ of the entire N\'eron--Severi group with respect to the Beauville--Bogomolov form. By construction, $H^2_{\mathrm{tr}}(X,\QQ)$ is an irreducible Hodge structure of K3-type, and we may consider the Kuga--Satake variety $\mathrm{KS}'(X)$ built from $H^2_{\mathrm{tr}}(X,\QQ)$. Then $\mathrm{KS}(X)$ is isogenous to a power of $\mathrm{KS}'(X)$, and Conjecture \ref{conj:KSHC} is equivalent to the analogous statement for $H^2_{\mathrm{tr}}(X,\QQ)\hookrightarrow H^2(\mathrm{KS}'(X)\times \mathrm{KS}'(X),\QQ)$;  see for example \cite[Chapter 4, \textbf{2.5}]{huyK3}. 
	\end{remark}

	Notice that there is no reference to a specific embedding in Conjecture \ref{conj:KSHC}; nevertheless, the Kuga--Satake construction yields a natural one.
	A more precise version of the Kuga--Satake--Hodge Conjecture \ref{conj:KSHC} is the following. 
	\begin{conjecture}
		\label{conj:refinedKSHC} 
		Keep the assumptions as in Conjecture \ref{conj:KSHC}.
		There exists an algebraic class $\gamma\in H^{4n-2}(X,\QQ) \otimes H^2(\mathrm{KS}(X)^2,\QQ) \subset H^{4n}(X\times \mathrm{KS}(X)^2, \QQ)$ which induces, as a correspondence, the Kuga--Satake embedding 
		\[
		\mu \colon H^2(X,\QQ)_{\mathrm{prim}}\hookrightarrow H^2(\mathrm{KS}(X)^2,\QQ).
		\]
	\end{conjecture} 
	
	The conjecture does not depend on the choice of $v_0$ in the construction of $\mu$.
	It is not clear to us whether Conjecture \ref{conj:KSHC} implies the more precise version Conjecture~\ref{conj:refinedKSHC} in general. 
	While Conjecture \ref{conj:KSHC} is sufficient for most purposes, in some cases the algebraicity of the specific Kuga--Satake embedding can be useful; see for instance the proof of the following result.
	
	\begin{theorem}[{\cite[Theorem 3.5]{floccari25}}]\label{thm:HCpowersKS}
		Let $(S,h)$ be a polarized K3 surface, and let $\mathrm{KS}(S)$ be the Kuga--Satake variety constructed from $H^2(S,\QQ)_{\mathrm{prim}}$. Assume that the Hodge conjecture holds for all powers of $S$ and that Conjecture \ref{conj:refinedKSHC} holds for $S$. Then, the Hodge conjecture holds for all powers of $\mathrm{KS}(S)$.
	\end{theorem}

	Determining the Kuga--Satake variety of a polarized hyper-K\"ahler variety $X$ is challenging in general. Descriptions of $\mathrm{KS}(X)$ in terms of the geometry of $X$ are available only in a few cases, for example, for abelian surfaces and Kummer K3 surfaces by Morrison \cite{morrison1985}, for double covers of $\PP^2$ ramified along six lines by Paranjape \cite{paranjape}, and for hyper-K\"ahler varieties of generalized Kummer type by Markman \cite{markman2019monodromy} and O'Grady~\cite{O'G21}. 
	
	The following result of Lombardo \cite{Lombardo} allows one to study abelian fourfolds of Weil type with discriminant 1 as isogeny factors of Kuga--Satake varieties (see also \cite[Theorem 9.2]{vanGeemen} and \cite[Theorem 4.1]{varesco}). 
	
	\begin{theorem}[Lombardo]\label{thm:Lombardo}
		Let $d\in \ZZ_{>0}$. Let $B$ be an abelian fourfold of $\QQ(\sqrt{-d})$-Weil type with discriminant 1. Then $B^4$ is isogenous to the Kuga--Satake variety of a 6-dimensional polarized Hodge structure of $\mathrm{K}3$-type $(V, q)$ such that
		\begin{equation*}
			\label{eqn:LombardoForm}
			(V, q)\cong \mathrm{U}^{\oplus 2}_{\QQ} \oplus \langle a \rangle \oplus \langle b \rangle,
		\end{equation*}
		with $a, b\in \ZZ_{<0}$, where $\mathrm{U}$ denotes a hyperbolic plane.
		
		Conversely, given a polarized Hodge structure of $\mathrm{K}3$-type $(V, q)$ with an isomorphism of quadratic spaces as above, the Kuga--Satake variety $\mathrm{KS}(V)$ is isogenous to $B^4$ for some abelian fourfold $B$ of $\QQ(\sqrt{-ab})$-Weil type with discriminant 1.
	\end{theorem}
	
	By the Torelli theorem for K3 surfaces, any $6$-dimensional Hodge structure $(V,q)$ as in Theorem \ref{thm:Lombardo} arises as direct summand of the Hodge structure of a projective K3 surface (\cite[Corollary 2.10]{Morrison1984}). A crucial tool that we use in this paper is the following theorem of Varesco \cite[Proposition 0.4]{varesco}.
	
	\begin{theorem}[Varesco]\label{thm:Varesco}
		Let $S$ be a very general $\mathrm{K}3$ surface of Picard number 16 with transcendental cohomology $H^2_{\mathrm{tr}}(S, \QQ)$ isometric to $\mathrm{U}^{\oplus 2}_{\QQ} \oplus \langle a \rangle \oplus \langle b \rangle$, with $a, b\in \ZZ_{<0}$. Let $B$ be the simple factor of $\mathrm{KS}(S)$ as in Theorem \ref{thm:Lombardo}. If the Kuga--Satake--Hodge Conjecture \ref{conj:KSHC} holds for $S$, then the Hodge conjecture holds for all powers of $B$.    
	\end{theorem}

	\section{Singular OG6-varieties and their Kuga--Satake varieties}

	Let $A$ be an abelian surface. The Mukai lattice of $A$ is the even cohomology $H^{2\bullet}(A,\ZZ)$ equipped with the pairing $$((a,b,c), (a',b',c'))=(b,b')-ac'-a'c,$$
	where $a,a'\in H^0(A, \ZZ)$, $b,b'\in H^2(A, \ZZ)$ and $c, c'\in H^4(A, \ZZ)$.
	
	Given an algebraic Mukai vector $\mathsf v\in H^{2\bullet}(A,\ZZ)$ and a polarization $H$ on $A$, let $M_{A}(\mathsf{v})$ be the moduli space of $H$-semistable coherent sheaves on $A$ with Chern character $\mathsf{v}$ (see \cite{huybrechts2010geometry} for a reference); we will always assume that the polarization $H$ is $\mathsf{v}$-generic and omit it from the notation. 
	If not empty, $M_{A}(\mathsf{v})$ is a projective variety of dimension $\mathsf{v}^2+2$ and its smooth locus carries a symplectic form (\cite{Muk84}). Assume that $\mathsf{v}^2\geq 6$. The Albanese map is an isotrivial fibration $M_{A}(\mathsf{v})\to A\times \hat{A}$; we let $K_{A}(\mathsf{v})$ denote the fibre. 
	If $\mathsf{v}$ is primitive, then $M_{A}(\mathsf{v})$ is smooth, and $K_{A}(\mathsf{v})$ is a hyper-K\"ahler variety of generalized Kummer type of dimension $2n=\mathsf{v}^2-2$ (\cite{Yos01}).
	If $\mathsf{v}$ is not primitive, then $K_{A}(\mathsf{v})$ is singular and does not admit a crepant resolution in most cases (\cite{KLS06}).
	
	Consider the Mukai vector $(2,0,-2)\in H^{2\bullet}(A,\ZZ)$. In \cite{O'G99, O'G03}, O'Grady constructed a crepant resolution $\widetilde{K}_{A}(2,0,-2)\to {K}_{A}(2,0,-2)$, which gives a $6$-dimensional smooth hyper-K\"ahler variety that is not deformation equivalent to a Hilbert scheme of points on a K3 surface or to a generalized Kummer variety. Lehn and Sorger \cite{LS06} showed that O'Grady's resolution coincides with the blow-up of $K_{A}(2, 0,-2)$ along its singular locus. In fact, any (non-empty) moduli space of sheaves $K_{A}(\mathsf{v})$ with $\mathsf{v}=2\mathsf{v}_0$ for a primitive algebraic Mukai vector $\mathsf{v}_0$ of square $2$ admits a crepant resolution $\widetilde{K}_{A}(\mathsf{v})$ by a hyper-K\"ahler variety deformation equivalent to O'Grady's example (\cite{PR13}).
	
	A \textit{hyper-K\"ahler manifold of $\mathrm{OG}6$-type} is by definition a compact hyper-K\"ahler manifold which is deformation equivalent to O'Grady's crepant resolution $\widetilde{K}_A(2,0,-2)$; these manifolds form a holomorphic family of dimension $6$. We shall be concerned with a subfamily of codimension $1$, the family of OG6-resolutions.
	
	\begin{definition}[Singular OG6-varieties]\label{def:SingularOG6}
		A \emph{singular $\mathrm{OG}6$-variety} $K$ is a compact K\"ahler complex analytic space which is a locally trivial deformation of $K_A(2,0,-2)$. Here, locally trivial deformation is as in the sense of  \cite{FlennerKosarew}. 
		An \emph{$\mathrm{OG}6$-resolution}  is a compact hyper-K\"ahler manifold isomorphic to the crepant resolution $\widetilde{K}$ of a singular $\mathrm{OG}6$-variety $K$ obtained by blowing up its singular locus. 
	\end{definition}
	
	The theory of locally trivial deformations of singular hyper-K\"ahler varieties admitting a crepant resolution has been developed in  \cite{BakkerLehn2021}, and it is entirely analogous to that of hyper-K\"ahler manifolds. The OG6-resolutions are exactly the deformations of $\widetilde{K}_A(2,0,-2)$ on which the parallel transport of the cohomology class of the exceptional divisor of $\widetilde{K}_A(2,0,-2)\to K_A(2,0,-2)$ remains algebraic; thus, OG6-resolutions form a~$5$-dimensional holomorphic subfamily among all hyper-K\"ahler manifolds of OG6-type. Projective OG6-resolutions come in $4$-dimensional families; as abelian surfaces only give $3$-dimensional families, the very general projective OG6-resolution cannot be realized from a singular moduli space of sheaves on an abelian surface.
	
	\begin{remark}\label{rmk:1}
		The pull-back along the resolution $\nu\colon \widetilde{K}\to K$ induces an embedding $\nu^*\colon H^2(K,\ZZ)\hookrightarrow H^2(\widetilde{K},\ZZ)$, by \cite[Theorem 1.7]{PR13}. The cohomology class of the exceptional divisor $E$ of $\nu$ is divisible by $2$, and we have $H^2(\widetilde{K},\ZZ)=\nu^*(H^2(K,\ZZ))\oplus \ZZ\cdot \tfrac{1}{2}[E]$. In particular, $H^2(K,\ZZ)$ carries a pure weight-$2$ Hodge structure of K3-type.
		We equip $H^2(K,\ZZ)$ with the pairing induced by the restriction of the Beauville--Bogomolov form on $H^2(\widetilde{K},\ZZ)$. As a lattice, $H^2(K,\ZZ)$ (resp. $H^2(\widetilde{K},\ZZ)$) is isometric to $\mathrm{U}^{\oplus 3}\oplus \langle -2\rangle$ (resp. to $\mathrm{U}^{\oplus 3}\oplus \langle -2\rangle \oplus \langle -2\rangle$), see \cite{rapagnetta2007topological}.
	\end{remark}
	
	We will study the Kuga--Satake varieties of singular OG6-varieties and their resolutions. Let $(K,h)$ be a quasi-polarized singular OG6-variety. We write $H^2(K,\ZZ)_{\mathrm{prim}}$ for the primitive cohomology (Remark \ref{rmk:quasi-polarized}); if $\nu\colon \widetilde{K}\to K$ is the OG6-resolution of $K$, then $\nu^*(H^2(K,\ZZ)_{\mathrm{prim}})$ is the orthogonal complement of the $2$-dimensional sublattice $\langle \nu^*(h), [E]\rangle$ of $H^2(\widetilde{K},\ZZ)$.
	
	\begin{remark}\label{rmk:quadratic_space}
		Over the rational numbers, $H^2(K,\QQ)_{\mathrm{prim}}$ is isometric to the quadratic space $\mathrm{U}^{\oplus 2}_{\mathbb{Q}} \oplus \langle -2\rangle \oplus \langle -2d\rangle$, where $2d $ is the Beauville--Bogomolov degree of $h$.
	\end{remark}
	
	The following theorem shows that singular OG6-varieties are naturally related to abelian fourfolds of Weil type with discriminant $1$.
	
	\begin{theorem} \label{thm:2}
		Let $(K, h)$ be a quasi-polarized singular $\mathrm{OG}6$-variety, with $h$ of Beauville--Bogomolov square $2d$. 
		\begin{enumerate}[label=(\roman*)]
			\item The singular locus $\Sigma$ of $K$ is isomorphic to $B_K /\pm 1$ for an abelian fourfold $B_K$ of $\QQ(\sqrt{-d})$-Weil type with discriminant $1$.
			\item The Kuga--Satake variety $\mathrm{KS}(H^2(K,\QQ)_{\mathrm{prim}})$ of $K$ is isogenous to $B_K^{ 4}$.
		\end{enumerate}
	\end{theorem}
	\begin{proof}
		The singular locus of O'Grady's moduli space $K_{A}(2,0,-2)$ parametrizes poly-stable sheaves of the form $E\oplus E'\in K_{A}(2,0,-2)$ with $E, E'\in M_{A}(1,0,-1)$, hence is isomorphic to the Albanese fiber of $\operatorname{Sym}^2(M_{A}(1,0,-1))$. As $M_{A}(1,0,-1)\simeq A\times \hat{A}$, the singular locus of $K_{A}(2,0,-2)$ is isomorphic to $(A\times \hat{A})/\pm 1$; see \cite{O'G03}. Thus, by construction, the singular locus of a (not necessarily projective) singular OG6-variety $K$ is a locally trivial deformation of $(A\times\hat{A})/\pm 1$. But, since $-1$ acts trivially on $H^2(A\times \hat{A},\ZZ)$, any such deformation is of the form $B/\pm 1$ for some complex torus~$B$ of dimension $4$ (see \cite[Corollary 3.6]{fujiki1983}). Hence, the singular locus of $K$ is isomorphic to $B_K/\pm 1$ for some complex torus $B_K$ of dimension $4$; if $K$ is projective, then $B_K$ is an abelian fourfold. 
		
		Next, for any singular OG6-variety $K$, the restriction $H^2(K,\QQ)\to H^2(\Sigma, \QQ)$ is injective, as was already observed in \cite{bertiniGiovenzana}. The main point is that the restriction of the holomorphic $2$-form of $K$ to its singular locus $\Sigma$ is non-zero. Hence, as for a very general (non-projective) singular OG6-variety $K$ the Hodge structure $H^2(K,\QQ)$ is irreducible, the restriction  $H^2(K,\QQ)\to H^2(\Sigma,\QQ)$ must be an embedding of Hodge structures for very general $K$. Since the injectivity of the induced map in cohomology is a topological property, therefore invariant under locally trivial deformations, we conclude that the restriction map is injective for any singular OG6-variety $K$. There is an isomorphism of Hodge structures $H^2(B_K/\pm 1, \QQ)\cong H^2(B_K,\QQ)$; by the above, we thus obtain an embedding of Hodge structures $$\psi\colon H^2(K,\QQ)\hookrightarrow H^2(B_K,\QQ).$$
		
		Assume now that $(K,h)$ is quasi-polarized of degree $2d$. We adapt arguments of O'Grady from \cite{O'G21} to complete the proof. It suffices to prove $(i)$ and $(ii)$ for a very general polarized singular OG6-variety $(K,h)$. In this case, $H^2({K},\QQ)_{\mathrm{prim}}$ is an irreducible Hodge structure, and, as a quadratic space, it is isometric to $\mathrm{U}_{\QQ}^{\oplus 2}\oplus \langle -2\rangle \oplus \langle -2d\rangle$, where $2d$ is the Beauville--Bogomolov degree of $h$ (Remark \ref{rmk:quadratic_space}). By Lombardo's result (Theorem \ref{thm:Lombardo}), the Kuga--Satake variety built from $H^2(K,\QQ)_{\mathrm{prim}}$ is isogenous to $C^{4}$, where $C$ is a simple abelian fourfold of $\QQ(\sqrt{-d})$-Weil type with discriminant $1$. By the universal property of the Kuga--Satake construction established by van Geemen--Voisin \cite{vanGeemenVoisin} and Charles \cite{charles2022}, the existence of the embedding of Hodge structures $\psi$ implies that $B_K$ is isogenous to a subquotient of the Kuga--Satake variety $C^{4}$. But since $C$ is simple and $B_K$ has dimension $4$, we conclude that $B_K $ is isogenous to $C$.
	\end{proof}
	
	Given a family $\mathcal{K}\to T$ of quasi-polarized singular OG6-varieties, we obtain a family $\mathcal{B}\to T$ of abelian fourfolds of Weil type with discriminant $1$, such that the singular locus of the fibre $\mathcal{K}_t$ is isomorphic to $\mathcal{B}_t/\pm 1$.
	Together with the next proposition, Theorem \ref{thm:2} implies Theorem \ref{thm:KSsingularOG6} from the introduction.
	
	\begin{proposition}\label{prop:completeFamilies}
		For $(K,h)$ varying in a complete family $\mathcal{K}\to T$ of quasi-polarized singular $\mathrm{OG}6$-varieties of degree $2d$, the family $\mathcal{B}\to T$ of abelian fourfolds of $\QQ(\sqrt{-d})$-Weil type with discriminant $1$ is complete up to isogeny.
	\end{proposition}
	\begin{proof}
		Let $\Lambda_{\mathrm{OG}6^{\mathrm{sing}}}=\mathrm{U}^{\oplus 3}\oplus \langle -2\rangle$ be the abstract lattice underlying the second cohomology of singular OG6-varieties, equipped with the Beauville--Bogomolov form. The divisibility $\mathrm{div}(h)$ of a class $h\in \Lambda_{\mathrm{OG}6^{\mathrm{sing}}}$ is the integer such that $(h, -) =\mathrm{div}(h) \cdot \ZZ$. 
		For a primitive $h\in \Lambda_{\mathrm{OG}6^{\mathrm{sing}}}$, we have $\mathrm{div}(h)=1$ or $2$. By Eichler's criterion, the orbit of~$h$ under the orthogonal group $O(\Lambda_{\mathrm{OG}6^{\mathrm{sing}}})$ is uniquely determined by its square and its divisibility.
		Let $h$ be a primitive class of square $2d$, and let $L\subset \Lambda_{\mathrm{OG}6^{\mathrm{sing}}}$ be the orthogonal complement to $h$. If $\mathrm{div}(h)=1$, then $h^{\bot}$ is isometric to $\mathrm{U}^{\oplus 2}\oplus \langle -2\rangle \oplus \langle -2d\rangle$, while if $\mathrm{div}(h)=2$ we must have $d=4d'-1$ for some integer $d'$, and then $h^{\bot}$ is isometric to $\mathrm{U}^{\oplus 2} \oplus \left(\begin{smallmatrix} -2 & -1 \\ -1 & -2d' \end{smallmatrix}\right)$. 
		In any case, $h^\bot\otimes_{\ZZ}\QQ$ is isometric to $\mathrm{U}_{\QQ}^{\oplus 2}\oplus \langle -2\rangle \oplus \langle -2d\rangle$.
		
		Let $L$ denote the lattice $H^2(K_t, \ZZ)_{\mathrm{prim}}$ for a very general fibre $(\mathcal{K}_t, h_t)$ of a family $\mathcal{K}\to T$ as above. By the surjectivity of the period map for singular OG6-varieties (\cite{BakkerLehn2021}), any Hodge structure of K3-type on the lattice $L$ is realized as primitive cohomology $H^2(\mathcal{K}_t, \ZZ)_{\mathrm{prim}}$ of a quasi-polarized singular OG6-variety of degree $2d$, and, hence, as primitive cohomology of a fibre $(\mathcal{K}_t, h_t)$ of $\mathcal{K}\to T$. By Theorem \ref{thm:2} it suffices to show that, given any abelian fourfold $B$ of $\QQ(\sqrt{-d})$-Weil type with discriminant $1$, there exists a Hodge structure of K3-type $L_B$ on the lattice $L$ such that $\mathrm{KS}(L_B)$ is isogenous to $B^4$.
		
		By Lombardo's Theorem \ref{thm:Lombardo}, there exists a polarized $\QQ$-Hodge structure of K3-type $(V,q)$ such that $\mathrm{KS}(V) $ is isogenous to $B^{4}$, with $(V,q)$ isometric to $\mathrm{U}^{\oplus 2}_{\QQ}\oplus \langle a\rangle \oplus \langle b\rangle$ for negative integers $a,b$ such that $ab = c^2\cdot d$ for some  $c\in \QQ^*$. Let $(V',q')$ be the same Hodge structure as $V$, but with the form multiplied by $-2a$; by Remark \ref{rmk:functorialityKS}, the Kuga--Satake variety $\mathrm{KS}(V')$ is also isogenous to $B^4$. 
		Notice that $(V',q')$ is isometric to $\mathrm{U}_{\QQ}^{\oplus 2}\oplus \langle -2\rangle \oplus \langle -2d\rangle$. Hence, there exists an isometry $V'\cong L_{\otimes_{\ZZ}}\QQ$; the desired Hodge structure $L_B$ is the one induced on $L$ by that on $V'$ via this isometry.
	\end{proof}
	
	\section{Proof of Theorem \ref{thm:1} and Theorem \ref{thm:1+}}
	
	Thanks to the work of Mongardi--Rapagnetta--Sacc\`a \cite{MRS18}, any OG6-resolution $\widetilde{K}$ admits a rational double cover $Z_K \dashrightarrow \widetilde{K}$ from a manifold $Z_K$ of $\mathrm{K}3^{[3]}$-type. The existence of this double cover comes from the fact that the class of the exceptional divisor of $\widetilde{K}\to K$ is divisible by $2$ in the Picard group of $\widetilde{K}$ (\cite{rapagnetta2007topological}).
	We recall that a manifold of $\mathrm{K}3^{[n]}$-type is a compact K\"ahler manifold which is deformation equivalent to the Hilbert scheme $S^{[n]}$ of $n$ points on a K3 surface $S$; manifolds of $\mathrm{K}3^{[n]}$-type are hyper-K\"ahler of dimension $2n$ (\cite{beauville1983varietes}). For $n\geq 2$, the second cohomology of manifolds of $\mathrm{K}3^{[n]}$-type is isometric to the orthogonal direct sum of the K3 lattice and a class of square $-2(n-1)$.

	Let $K$ be a singular OG6-variety (Definition \ref{def:SingularOG6}), and denote by $\Sigma$ (resp. $\Omega$) the singular locus of $K$ (resp. of $\Sigma$); then $\Sigma\cong B_K/\pm 1$ for the $4$-dimensional complex torus $B_K$ of Theorem \ref{thm:2}, and $\Omega$ consists of the $256$ nodes of $\Sigma$.
	
	\begin{theorem}[Mongardi--Rapagnetta--Sacc\`a] \label{thm:3}
		Let $K$ be any singular $\mathrm{OG}6$-variety. Then there exists a manifold $Z_K$ of $\mathrm{K}3^{[3]}$-type and a degree-2 generically finite morphism $\phi\colon Z_K\to K$ such that: $\phi_{|_{\phi^{-1}}(K\setminus \Sigma)}$ is an \'etale double cover, $\Delta\coloneqq \phi^{-1}(\Sigma)$ is isomorphic to $\mathrm{Bl}_{\Omega}(\Sigma)$, and $\phi_{|_{\Delta}}\colon \Delta\to \Sigma$ is identified with the blow-up map. 
	\end{theorem} 
	\begin{proof}
		In \cite{MRS18}, the theorem is stated only when $K$ is the Albanese fibre of a moduli space $M_{A}(\mathsf{v})$ of sheaves on an abelian surface $A$. Not all singular OG6-varieties may be constructed in this way, but it is immediate to deduce the stronger statement using locally trivial deformations as in \cite[Proof of Proposition 5.3]{MRS18}. 
	\end{proof}
	
	\begin{remark}\label{rmk:2}
		The map $\phi\colon Z_K\to K$ induces a Hodge isometry $$\phi_*\colon H^2_{\mathrm{tr}}(Z_K,\QQ)(2)\xrightarrow{\ \sim \ }H^2_{\mathrm{tr}}(K,\QQ),$$ where the form on the left-hand side is multiplied by $2$ (see \cite[Lemma~3.4]{floccari22}), and $H^2_{\mathrm{tr}}(-, \QQ)$ stands for the transcendental part of the second cohomology (i.e.~the minimal sub-Hodge structure of $H^2(-, \QQ)$ containing $H^{2,0}$). Hence, the Kuga--Satake varieties of $Z_K$ and $K$ share the same isogeny factors, by Remark \ref{rmk:functorialityKS} and Remark \ref{rmk:KStr}.
	\end{remark} 
	
	The best known procedure, due to Mukai \cite{mukai1987moduli}, to construct varieties of $\mathrm{K}3^{[n]}$-type is by taking moduli spaces of stable sheaves on K3 surfaces; see for instance \cite{huybrechts2010geometry} for more details.
	
	\begin{corollary}\label{cor:1}
		Let $K$ be a projective singular $\mathrm{OG}6$-variety. Then the Mongardi--Rapagnetta--Sacc\`a double cover $Z_K$ is a projective variety of $\mathrm{K}3^{[3]}$-type which is birational to a moduli space of stable sheaves on a $\mathrm{K}3$ surface $S_K$. Moreover, $S_K$ is uniquely determined up to isomorphism.
	\end{corollary}
	\begin{proof}
		The proof is exactly the same as that of \cite[Proposition 3.3]{floccari22}.
		By Theorem \ref{rmk:2}, the transcendental lattice of $Z_K$ is of rank $k\leq 6$ and has signature $(2, k-2)$. Hence, by the Torelli theorem there exists a K3 surface $S_K$ such that $H^2_{\mathrm{tr}}(S_K,\ZZ)$ is Hodge isometric to $H^2_{\mathrm{tr}}(Z_K,\ZZ)$ (see \cite[Corollary 2.10]{Morrison1984}). Then it follows from \cite[Proposition 4]{Addington2016} that $Z_K$ is birational to a smooth and projective moduli space $M_{S_K}(\mathsf{v})$ of stable sheaves, for some Mukai vector $\mathsf{v}$ and a generic polarization. 
		Since $S_K$ is projective of Picard rank at least $16$, it is determined up to isomorphism by its transcendental lattice, by \cite[Chapter 16, Corollary 3.8]{huyK3}. 
	\end{proof}
	
	\begin{remark}
		By \cite[Proposition 2.3]{mongardiwandel}, we can realize $Z_K$ as a moduli space of objects in the derived category of $S_K$ which are stable with respect to a generic Bridgeland stability condition (\cite{Bri08, BM14b}).
	\end{remark}

	Now let $K$ be a projective singular OG6-variety, let $\phi\colon Z_K\to K$ be the morphism given in Theorem \ref{thm:3}, and let $S_K$ be the K3 surface given by Corollary \ref{cor:1}.

	\begin{theorem}\label{thm:4}
		For any projective singular $\mathrm{OG}6$-variety $K$, the Kuga--Satake--Hodge Conjecture \ref{conj:KSHC} holds for the $\mathrm{K}3$ surface $S_K$. 
	\end{theorem}
	\begin{proof}
		By specialization of cycles, it suffices to prove the theorem when $K$ is very general. Then $H^2_{\mathrm{tr}}(K,\ZZ)=H^2(K,\ZZ)_{\mathrm{prim}}$ is of rank $6$. By Remark \ref{rmk:2}, the Kuga--Satake variety of $Z_K$ is isogenous to a power of that of $K$; by Theorem \ref{thm:2}, the Kuga--Satake variety $\mathrm{KS}(H^2_{\mathrm{tr}}(Z_K,\QQ))$ is in this case isogenous to $B_K^{4}$, where $B_K$ is the abelian variety such that the singular locus of $K$ is isomorphic to $B_K/\pm1$.
		
		By Theorem \ref{thm:3}, there is a closed embedding $j\colon \mathrm{Bl}_{\Omega}(\Sigma)\hookrightarrow Z_K$, where $\Sigma\cong B_K/\pm 1$ and $\Omega$ is the singular locus of $\Sigma$, which is the image of the subset $B_K[2]$ of points of order $2$ in $B_K$. Then, $\mathrm{Bl}_{\Omega}(\Sigma)\cong  (\mathrm{Bl}_{B_K[2]}(B_K))/\pm 1$ and there is a commutative diagram 
		\[
		\begin{tikzcd}
			& \mathrm{Bl}_{B_K[2]}(B_K) \arrow[swap]{ld}{p} \arrow{rd}{q} \\
			B_K \arrow{dr} && \mathrm{Bl}_{\Omega}(\Sigma)
			\arrow{dl} 
			\\
			& B_K/\pm 1
		\end{tikzcd}
		\]
		in which the maps going to the left are blow-ups and those going to the right are quotients with respect to the action of $-1$.
		We then get a morphism of Hodge structures
		\[\Psi\colon H^2_{\mathrm{tr}}(Z_K,\QQ) \to H^2(B_K,\QQ)\] by setting $\Psi \coloneqq p_* \circ q^* \circ j^*$. By construction, $\Psi$ is the restriction of the correspondence induced by an algebraic class $\gamma\in H^{\bullet}(Z_K\times B_K,\QQ)$. By \cite{CM13} and \cite{kleiman}, the Lefschetz standard conjecture holds for both the $\mathrm{K}3^{[3]}$-variety $Z_K$ and the abelian variety $B_K$; hence, also the K\"unneth standard conjecture holds for $Z_K$ and $B_K$ (\cite{kleiman}). It follows that the K\"unneth component $\bar{\gamma} \in H^{10}(Z_K,\QQ)\otimes H^2(B_K,\QQ)$ of the algebraic class $\gamma$ is still algebraic; therefore, $\Psi$ is the correspondence induced by the algebraic class $\bar{\gamma}$.
		
		We claim now that~$\Psi$ is injective, which will prove Conjecture \ref{conj:KSHC} for $Z_K$, as $\mathrm{KS}(H^2_{\mathrm{tr}}(Z_K,\QQ))$ is isogenous to $B_K^4$ (see Remark \ref{rmk:KStr}).
		Since~$H^2_{\mathrm{tr}}(Z_K,\QQ)$ is an irreducible Hodge structure, it is sufficient to show that $\Psi$ is not the zero map. Let $\sigma$ be a holomorphic symplectic $2$-form on $Z_K$. If $\sigma_{|_{\mathrm{im}(j)}}$ was identically $0$, then $\mathrm{im}(j)$ would have dimension at most $3$ (the dimension of an isotropic subspace of a 6-dimensional symplectic space is at most 3), but as $\dim (B_K)=4$, this is not the case. 
		Hence, $j^*(\sigma)$ is a non-zero holomorphic~$2$-form on~$ \mathrm{Bl}_{\Omega}(\Sigma)$. It is well-known that $p_*\circ q^*$ induces an isomorphism 
		$$p_*\circ q^* \colon H^0(\mathrm{Bl}_{\Omega}(\Sigma),\Omega^2) \xrightarrow{\ \sim \ } H^0(B_K,\Omega^2)$$ between the spaces of holomorphic $2$-forms. Indeed, $q$ is the quotient by the automorphism $-1$, whose action on $H^2(\mathrm{Bl}_{B_K[2]}(B_K),\QQ)$ (and so on $H^0(\mathrm{Bl}_{B_K[2]}(B_K),\Omega^2)$) is trivial; therefore, $q^*\colon H^0(\mathrm{Bl}_{\Omega}(\Sigma),\Omega^2)\to H^0(\mathrm{Bl}_{B_K[2]}(B_K),\Omega^2)$ is an isomorphism. Further, $p$ is birational, and so it induces an isomorphism between the spaces of holomorphic differential forms. Therefore, $\Psi(\sigma)$ is a non-zero holomorphic $2$-form on $B_K$, and we conclude that $\Psi$ is an embedding, as claimed.
		
		By Corollary \ref{cor:1}, there exists a projective K3 surface $S_K$ such that $Z_K$ is birational to a smooth and projective moduli space $M_{S_K}(\mathsf{v})$ of stable sheaves on $S_K$. 
		By \cite{mukai1987moduli, O'G97}, there exists a Hodge isometry $$\Phi\colon H^2_{\mathrm{tr}}(S_K,\QQ)\xrightarrow{\ \sim \ } H^2_{\mathrm{tr}}(M_{S_K}(\mathsf{v}),\QQ),$$ which is induced by an algebraic class $\delta\in H^2(S_K,\QQ)\otimes H^2(M_{S_K}(\mathsf{v}),\QQ)$ (see \cite[pp. 407]{floccariKum3} or \cite[\S5.3]{floccari25}). 
		Moreover, by \cite[Lemma 2.6]{Huy99} a birational map $f\colon M_{S_K}(\mathsf{v})\dashrightarrow Z_K$ induces a Hodge isometry $$f_*\colon H^2(M_{S_K}(\mathsf{v}),\QQ)\xrightarrow{\ \sim \ } H^2(Z_K,\QQ);$$
		since the standard conjectures hold for varieties of $\mathrm{K}3^{[3]}$-type by \cite{CM13}, the isometry $f_*$ is induced by an algebraic class $\eta\in H^{10}(M_{S_K}(\mathsf{v}), \QQ) \otimes H^2(Z_K,\QQ)$.
		Therefore, the composition $\Psi\circ f_*\circ \Phi$ gives an embedding of Hodge structures
		$$\Psi\circ f_*\circ \Phi\colon H^2_{\mathrm{tr}}(S_K,\QQ)\hookrightarrow H^2(B_K,\QQ)$$
		which is induced by an algebraic class in $H^2(S_K,\QQ)\otimes H^2(B_K,\QQ)$. 
		Since the Kuga--Satake variety of $S_K$ is isogenous to a power of that of $K$, and, hence, to a power of~$B_K$, we have shown that the Kuga--Satake--Hodge Conjecture \ref{conj:KSHC} holds for $S_K$.
	\end{proof}
	
	We can now complete the proof of our main results.
	\begin{proof}[Proof of Theorem \ref{thm:1}]
		Let $B$ be a very general abelian fourfold of Weil type with discriminant $1$. By Proposition \ref{prop:completeFamilies}, there exists a singular OG6-variety $K$ such that $B$ is isogenous to the fourfold $B_K$ associated with $K$ as in Theorem~\ref{thm:2}. By Theorem~\ref{thm:3} and Corollary \ref{cor:1}, there exists a K3 surface $S_K$ whose Kuga--Satake variety is isogenous to a power of~$B_K$. Moreover, by Theorem \ref{thm:4}, Conjecture \ref{conj:KSHC} holds for $S_K$. Since~$B$ is very general, we can now apply Varesco's Theorem \ref{thm:Varesco} to conclude that the Hodge conjecture holds for $B$ (and all of its powers). In particular, the Hodge--Weil classes are algebraic. Via specialization, it follows that the Hodge--Weil classes are algebraic for any abelian fourfold of Weil type with discriminant $1$.	
	\end{proof} 
	
	By work of Varesco, Theorem \ref{thm:4} has the following consequence. 
	\begin{corollary}\label{cor:HCpowersK3}
		For any projective singular $\mathrm{OG}6$-variety $K$, the Hodge conjecture holds for all powers of the $\mathrm{K}3$ surface $S_K$.
	\end{corollary}
	\begin{proof}
		Starting from a complete polarized family $\mathcal{K}\to T$ of singular OG6-varieties, we obtain a $4$-dimensional family of K3 surfaces $\mathcal{S}\to T$, of general Picard rank $16$, such that the Kuga--Satake variety of a very general fibre $\mathcal{S}_t$ is a power of a simple abelian fourfold of Weil type with discriminant $1$. By Theorem \ref{thm:4}, the Kuga--Satake--Hodge conjecture holds for the K3 surfaces $\mathcal{S}_t$ for any $t\in T$. By \cite[Theorem 0.2]{varesco}, the Hodge conjecture holds for any power of the K3 surface $\mathcal{S}_t$, for any $t\in T$.
	\end{proof}
	
	We will now prove Theorem \ref{thm:1+} using the criterion given in Theorem \ref{thm:HCpowersKS}. We first bootstrap from Theorem \ref{thm:4} to establish the stronger version of the  Kuga--Satake--Hodge conjecture as stated in Conjecture \ref{conj:refinedKSHC}.
	
	\begin{corollary}\label{cor:KSforS_K}
		Let $K$ be a projective singular $\mathrm{OG}6$-variety. Let $S_K$ be the $\mathrm{K}3$ surface associated with $K$. Then Conjecture \ref{conj:refinedKSHC} holds for $S_K$.
	\end{corollary}
	\begin{proof}
		As the Kuga--Satake construction as well as the definition of the Kuga--Satake correspondence $\mu$ can be done in families, via specialization of cycles we may assume that $K$ is a very general projective singular $\mathrm{OG}6$-variety. In this case, $S_K$ is of Picard rank $16$ and $\mathrm{KS}(H^2_{\mathrm{tr}}(S_K,\QQ))$ is isogenous to $B_K^4$ for a very general abelian fourfold $B_K$ of Weil type with discriminant $1$. By Theorem \ref{thm:4}, the Kuga--Satake--Hodge Conjecture~\ref{conj:KSHC} holds for $S_K$; by Varesco's Theorem \ref{thm:Varesco}, the Hodge conjecture holds for all powers of $B_K$. 
		Via a standard argument using homological motives (see Example \ref{example:1} $(iv)$ and Lemma \ref{lem:easy}), it follows that the Hodge conjecture holds for any variety isomorphic to a product of powers of $S_K$ and~$B_K$.
		Since $\mathrm{KS}(S_K)$ is isogenous to a power of $B_K$, the Kuga--Satake correspondence $\mu\colon H^2(S_K,\QQ)_{\mathrm{prim}} \hookrightarrow H^2(\mathrm{KS}(S_K)^2,\QQ)$ of Conjecture~\ref{conj:refinedKSHC} is algebraic, being induced by a Hodge class  $[\mu] \in H^2(S_K,\QQ)\otimes H^2(\mathrm{KS}(S_K)^2,\QQ)$.
	\end{proof}
	
	\begin{proof}[Proof of Theorem \ref{thm:1+}]
		Let $B$ be an abelian fourfold of Weil type with discriminant~$1$. By Theorem \ref{thm:2} and Proposition \ref{prop:completeFamilies}, there exists a polarized singular OG6-variety~$(K,h)$ whose Kuga--Satake variety is isogenous to $B^4$.  
		Let $S_K$ be the K3 surface associated with $K$ via Theorem \ref{thm:3} and Corollary \ref{cor:1}; the Kuga--Satake variety $\mathrm{KS}(S_K)$ of $S_K$ is isogenous to a power of $B$. The Hodge conjecture holds for any power of $S_K$ by Corollary \ref{cor:HCpowersK3}, and the Kuga--Satake correspondence of Conjecture \ref{conj:refinedKSHC} is algebraic for $S_K$ by Corollary \ref{cor:KSforS_K}. By Theorem \ref{thm:HCpowersKS}, the Hodge conjecture holds for all powers of the Kuga--Satake variety $\mathrm{KS}(S_K)$ of $S_K$, and, hence, for any power of $B$.
	\end{proof}
	
	\section{Applications}\label{sec:applications}
	In the final section we will deduce some consequences for algebraic cycles on the hyper-K\"ahler varieties appearing in this paper. We will work with the category $\mathsf{Mot}$ of homological motives with rational coefficients over $\CC$, see \cite{andre}. The objects of $\mathsf{Mot}$ are triples $(X,p,n)$, where $X$ is a smooth and projective variety, $p$ is an idempotent correspondence induced by an algebraic cohomology class in $H^{2\dim X}(X\times X,\QQ)$, and $n$ is an integer. The morphisms between motives $(X,p,n)$ and $(Y,q,m)$ are defined as the correspondences $f$ induced by an algebraic class in $H^{2\dim X -2n+2m}(X\times Y,\QQ)$ such that $f=f\circ p =q\circ f$.
	There is a natural functor attaching to a smooth projective variety~$X$ its motive $\mathsf{h}(X)\coloneqq (X, [\Delta_X], 0)$, where $\Delta_X\subset X\times X$ is the diagonal. We denote by $\mathbb{Q}(-i)$ the Tate motive $(\Spec{\CC},[\Delta], -i)$ of weight $2i$.
	There is also a faithful realization functor $R\colon \mathsf{Mot}\to \mathsf{HS}$ to the category of pure polarizable $\QQ$-Hodge structures which sends a motive $(X,p,n)$ to the Hodge structure $p_*(H^{\bullet}(X,\QQ)(n))$. The Hodge conjecture is the statement that the realization functor is full (see \cite[Proposition 7.2.1.3]{andre}).
	
	The category $\mathsf{Mot}$ is a pseudo-abelian rigid symmetric monoidal category. Grothendieck's standard conjectures \cite{grothendieck1969standard} would imply that $\mathsf{Mot}$ is a semisimple abelian category, thanks to work of Jannsen \cite{Jan92} and Andr\'e \cite{andre1996Motives}. 
	Given $\mathsf{m}\in \mathsf{Mot}$, let $\langle \mathsf{m}\rangle_{\mathsf{Mot}}$ be the pseudo-abelian tensor subcategory of $\mathsf{Mot}$ generated by $\mathsf{m}$; in other words, $\langle \mathsf{m}\rangle_{\mathsf{Mot}}$ is the subcategory of $\mathsf{Mot}$ generated from $\mathsf{m}$ by taking direct sums, tensor products, duals, and direct summands.
	
	\begin{theorem}[{\cite[Theorem 4]{Ara06}}]\label{thm:stdConj}
		The standard conjectures hold for a smooth and projective variety $X$ if and only if $\langle \h(X)\rangle_{\mathsf{Mot}}$ is a semisimple abelian category. 
	\end{theorem}
	
	\begin{remark}\label{rmk:stdConjectures}
		The standard conjectures hold for curves, surfaces and abelian varieties (\cite{kleiman}). They hold for varieties $X$ and $Y$ if and only if they hold for $X\times Y$. 
	\end{remark}
	
	\begin{definition}
		Let $X$ and $Y$ be smooth projective varieties. We say that $X$ is \textit{motivated} by $Y$ if the motive $\h(X)\in \mathsf{Mot}$ belongs to the subcategory $\langle \h(Y)\rangle_{\mathsf{Mot}}$.
	\end{definition}
	
	This notion has the following immediate implications.
	
	\begin{lemma}\label{lem:easy}
		Assume that $X$ is motivated by $Y$. If the standard conjectures hold for $Y$, then they hold for $X$; if the Hodge conjecture holds for all powers of $Y$, then it holds for $X$ and all of its powers.   
	\end{lemma}
	\begin{proof}
		These statements are immediate from the definition, using Theorem \ref{thm:stdConj} and that the Hodge conjecture for all powers of $X$ is equivalent to the fullness of the realization functor $R\colon \mathsf{Mot}\to \mathsf{HS}$ when restricted to $\langle \h(X)\rangle_{\mathsf{Mot}}$.
	\end{proof} 
	
	Let us recall some examples.
	
	\begin{example}\phantomsection \label{example:1}
		\begin{enumerate}[label=\textit{(\roman*)}]
			\item Let $f\colon Y\to X$ be a dominant morphism of smooth projective varieties. Then $X$ is motivated by $Y$. 
			\item Let $X$ be a smooth projective variety and $Z\subset X$ a smooth closed subvariety. Let $Y\to X$ be the blow-up of $X$ along $Z$. Then $Y$ is motivated by $X\times Z$.
			\item Projective spaces are motivated by any smooth and projective variety $X$. Indeed, an ample divisor on $X$ yields a split submotive $\mathsf{Q}(-1) \subset \h(X)$. Thus, all Tate motives are contained in $\langle \h(X)\rangle_{\mathsf{Mot}}$. Since $\h(\mathbb{P}^n)$ is a sum of Tate motives, it belongs to $\langle \h(X)\rangle_{\mathsf{Mot}}$.
			\item Let $S$ be a projective K3 surface and assume that the Kuga--Satake--Hodge Conjecture \ref{conj:KSHC} holds for $S$. Then $S$ is motivated by its Kuga--Satake variety $\mathrm{KS}(S)$. Indeed, as shown in \cite{Kahn-Murre-Pedrini}, $\h(S)$ is the sum of its transcendental part $\h_{\mathrm{tr}}(S)$ and Tate motives; by the above, it suffices to show that $\h_{\mathrm{tr}}(S)\in \langle \h(\mathrm{KS}(S))\rangle_{\mathsf{Mot}}$. By assumption, there exists a morphism $\iota\colon \h_{\mathrm{tr}}(S)\to \h(\mathrm{KS}(S)^2)$ in $\mathsf{Mot}$, whose realization $R(\iota)\colon H^2_{\mathrm{tr}}(S,\QQ)\hookrightarrow H^2(\mathrm{KS}(S)^2,\QQ)$ is an embedding of Hodge structures. The standard conjectures hold for $S\times \mathrm{KS}(S)^2$ (see Remark \ref{rmk:stdConjectures}); by Theorem \ref{thm:stdConj}, the subcategory $\langle S\times \mathrm{KS}(S)^2\rangle_{\mathsf{Mot}}$ is abelian and semisimple. Hence, there exists motives $\mathsf{ker}(\iota)$ and $\mathsf{im}(\iota)$ which are the kernel and the image of $\iota$ respectively. Since the realization of $\iota$ is injective, the Hodge realization of $\mathsf{ker}(\iota)$ is zero; then $\mathsf{ker}(\iota)$ is the zero motive and $\iota\colon \h_{\mathrm{tr}}(S)\to \h(\mathrm{KS}(S)^2)$ is a split monomorphism in $\mathsf{Mot}$.
			\item Let $S$ be a projective K3 surface and let $M_{S}(\mathsf{v})$ be a smooth and projective moduli space of stable sheaves on $S$ with Mukai vector $\mathsf{v}$. Then B\"ulles proved in \cite{Bue18} that the variety of $\mathrm{K}3^{[n]}$-type $M_{S}(\mathsf{v})$ is motivated by the K3 surface $S$.
		\end{enumerate}
	\end{example}

	It is generally expected that the motive of a hyper-K\"ahler variety $X$ should be controlled by its second cohomology $H^2(X,\QQ)$.
	We make this precise through the following conjecture.
	
	\begin{conjecture}\label{conj:motivesHK}
		\phantomsection
		\begin{enumerate}[label=(\roman*)]
			\item Let $X$ be a hyper-K\"ahler variety and let $\mathrm{KS}(X)$ be its Kuga-Satake variety. Then $X$ is motivated by $\mathrm{KS}(X)$.
			\item Let $X$ and $Y$ be deformation equivalent hyper-K\"ahler varieties. Assume that there exists a Hodge isometry $H^2(X,\QQ)\xrightarrow{\ \sim \ } H^2(Y,\QQ)$. Then the motives of $X$ and $Y$ are isomorphic. 
		\end{enumerate}
	\end{conjecture} 
	Recall that the abelian variety $\mathrm{KS}(X)$ is built using only $H^2(X,\QQ)$, so that $(i)$ says that this cohomology group controls the whole motive of $X$ in some way. 
	At the level of Hodge structures, Conjecture \ref{conj:motivesHK} is known: in fact, the Hodge structure $H^{\bullet}(X,\QQ)$ belongs to the tensor subcategory of Hodge structures generated by $H^{\bullet}(\mathrm{KS}(X),\QQ)$ by \cite{KSV2017} or \cite{FFZ}, and deformation equivalent hyper-K\"ahler varieties with Hodge isometric second cohomology have Hodge isomorphic cohomology algebras (\cite{solda19}). 
	
	Conjecture \ref{conj:motivesHK} may be formulated in various categories of motives; the strongest statement would be in the setting of Chow motives. In the realm of Andr\'e motives, Conjecture~\ref{conj:motivesHK}~$(i)$ has been confirmed for all hyper-K\"ahler varieties of known deformation type (\cite{soldatenkov19, FFZ}), while $(ii)$ has been proven for all hyper-K\"ahler varieties with second Betti number at least $7$ in \cite{floccari2020} (this assumption is satisfied by all known examples).
	We will focus here on the intermediate case of homological motives.
	
	Let $K$ be a projective singular OG6-variety. We have introduced the following smooth and projective varieties related to $K$: 
	\begin{itemize}
		\item the OG6-resolution $\widetilde{K}\to K$;
		\item the abelian fourfold $B_K$ such that $\Sigma:=\mathrm{Sing}(K)\cong B_K/\pm 1$, as in Theorem \ref{thm:2}; 
		\item the $\mathrm{K}3^{[3]}$-variety $Z_K$ arising as double cover of $K$ via Theorem \ref{thm:3};
		\item the K3 surface $S_K$ of Corollary \ref{cor:1} such that $Z_K$ is birational to a moduli space of stable sheaves $M_{S_K}(\mathsf{v})$ on $S_K$.
	\end{itemize} 
	
	Recall that, by Theorem \ref{thm:2}, the Kuga-Satake variety $\mathrm{KS}(K)$ is isogenous to $B_K^{4}$. Since we have Hodge isometries (see Remark \ref{rmk:2} and the proof of Theorem \ref{thm:4})
	\[ H_{\mathrm{tr}}^2(S_K,\QQ)(2)\xrightarrow{\ \sim \ } H_{\mathrm{tr}}^2(Z_K,\QQ)(2) \xrightarrow{\ \sim \ } H_{\mathrm{tr}}^2(K,\QQ),\]
	the Kuga-Satake varieties $\mathrm{KS}(Z_K)$ and $\mathrm{KS}(S_K)$ are isogenous to a power of $B_K$ as well (by Remark \ref{rmk:KStr}). 
	The following result establishes Conjecture \ref{conj:motivesHK} $(i)$ at the level of homological motives for each of the above varieties.
	
	\begin{theorem}\label{thm:motivated}
		Let $K$ be any projective singular $\mathrm{OG}6$-variety. Then the $\mathrm{OG}6$-resolution $\widetilde{K}$, the Mongardi--Rapagnetta--Sacc\`a double cover $Z_K$ and the $\mathrm{K}3$ surface $S_K$ are all motivated by the abelian fourfold $B_K$. 
	\end{theorem}
	\begin{proof}
		By Theorem \ref{thm:4}, the Kuga--Satake--Hodge Conjecture \ref{conj:KSHC} holds for $S_K$. By Example \ref{example:1} $(iv)$, this K3 surface is thus motivated by its Kuga--Satake variety, and, hence, $S_K$ is motivated by $B_K$.
		By construction, $Z_K$ is birational to a smooth and projective moduli space $M_{S_K}(\mathsf{v})$ of stable sheaves on $S_K$. By the work of Riess \cite{Rie14}, the motives of $Z_K$ and $M_{S_K}(\mathsf{v})$ are isomorphic. Hence, by \cite{Bue18}, the variety $Z_K$ is motivated by the K3 surface $S_K$ (see Example \ref{example:1} $(v)$), and, therefore, by $B_K$.  
		
		Let us now consider the OG6-resolution $\widetilde{K}$. The map $\phi\colon Z_K\to K$ of Theorem~\ref{thm:3} induces a rational map $\widetilde{\phi}\colon Z_K\dashrightarrow \widetilde{K}$ which is resolved as follows. Recall that $Z_K$ contains a subvariety $\Delta\cong \mathrm{Bl}_{B_K[2]} (B_K/\pm 1)$, and let $\Gamma\subset \Delta$ be the union of the $256$ copies of $\PP^3$ arising as exceptional divisor of the blow-up map $\Delta\to B_K/\pm 1$.
		Consider now the blow-up $Y\coloneqq \mathrm{Bl}_{\Gamma} (Z_K)$. By \cite[Remark 4.6]{MRS18}, the strict transform $\Delta'\subset Y$ of~$\Delta$ is again isomorphic to $\mathrm{Bl}_{B_K[2]} (B_K/\pm 1)$. Setting $\hat{Y}\coloneqq \mathrm{Bl}_{\Delta'}(Y)$, the rational map $\widetilde{\phi}$ extends to a regular morphism $\psi \colon \hat{Y}\to \widetilde{K}$ by \cite[Corollary 4.3]{MRS18}.
		
		By $(i)$ and $(ii)$ in Example \ref{example:1}, the OG6-resolution $\widetilde{K}$ is motivated by $\hat{Y}$, which in turn is motivated by $Z_K\times \Gamma\times \Delta'$. Hence, it will be sufficient to show that each of $Z_K$, $\Gamma$ and $\Delta$ is motivated by $B_K$. This is clear for $\Gamma$ which is a union of projective spaces, and was proven above for $Z_K$. Finally, $\Delta'\cong \mathrm{Bl}_{B_K[2]}(B_K/\pm 1)$ is motivated by $B_K$ as well, by Example \ref{example:1}.
	\end{proof}
	
	\begin{corollary}\label{cor:applicationHC}
		Let $K$ be a projective singular $\mathrm{OG}6$-variety, with associated varieties $B_K$, $\widetilde{K}$, $Z_K$ and $S_K$ as above. The Hodge conjecture holds for any variety $X$ isomorphic to a product of powers of $B_K$, $\widetilde{K}$, $Z_K$ and $S_K$.
	\end{corollary}
	\begin{proof}
		By Theorem \ref{thm:motivated}, the homological motive of $X$ belongs to $\langle \h(B_K)\rangle_{\mathsf{Mot}}$. By Theorem \ref{thm:1+}, the Hodge conjecture holds for any power of $B_K$, and, hence, for $X$, by Lemma \ref{lem:easy}.
	\end{proof}
	
	\begin{remark}\label{rmk:precisation}
		If $K$ and $K'$ are projective singular OG6-varieties such that there exists a Hodge similarity between $H^2_{\mathrm{tr}}(K,\QQ)$ and $H^2_{\mathrm{tr}}(K',\QQ)$, then $B_K$ and $B_{K'}$ are isogenous (see Remark~\ref{rmk:functorialityKS}), and in particular there is an isomorphism $\h(B_K)\cong \h(B_{K'})$ of homological motives. We can then conclude that the Hodge conjecture holds for any variety $X$ isomorphic to a product of powers of $B_K, \widetilde{K}, Z_{K}, S_{K}$ and of $B_{K'}, \widetilde{K}', Z_{K'}, S_{K'}$.
	\end{remark}
	
	A priori, there may exist non-isomorphic crepant resolutions of a singular OG6-variety, but any two such 
	resolutions are birational varieties of OG6-type and have isomorphic motive by \cite{Rie14} (her statement is in terms of Chow rings, but her proof constructs isomorphisms of Chow motives). 
	As special cases of Corollary \ref{cor:applicationHC}, any crepant resolution $\widetilde{K}\to K$ of a singular OG6-variety satisfies the Kuga--Satake--Hodge conjecture, and the Hodge conjecture holds for all powers of $\widetilde{K}$. Combined with Theorem \ref{thm:motivated}, this proves Theorem \ref{thm:HCOG6-resolutions} from the introduction. 
	
	As a consequence of Corollary \ref{cor:applicationHC}, we can show that Conjecture \ref{conj:motivesHK} $(ii)$ holds for the homological motives of OG6-resolutions. 
	\begin{corollary}
		Let $\widetilde{K}$ and $\widetilde{K}'$ be projective OG6-resolutions, and assume that there exists a Hodge isometry $H^2(K,\QQ)\xrightarrow{ \ \sim \ } H^2(K',\QQ)$. Then the homological motives of $K$ and $K'$ are isomorphic.
	\end{corollary}
	\begin{proof}
		By \cite{solda19} (see also \cite{floccari2020}), there exists an isomorphism of total Hodge structures $f\colon H^{\bullet}(\widetilde{K},\QQ)\xrightarrow{\ \sim \ } H^{\bullet}(\widetilde{K}',\QQ)$ (which in fact is an isomorphism of graded algebras). Then $f$ is induced by a Hodge class in $H^{12}(\widetilde{K}\times \widetilde{K}', \QQ)$. By Corollary \ref{cor:applicationHC} and Remark~\ref{rmk:precisation}, the Hodge conjecture holds for $\widetilde{K}\times \widetilde{K}'$ and therefore $f$ and its inverse are induced by algebraic correspondences; thus, $f$ is the realization of an isomorphism $\mathsf{f}\colon \h(\widetilde{K})\xrightarrow{ \ \sim \ } \h(\widetilde{K}')$ of homological motives.
	\end{proof}
	
	Let $L\subset \CC$ be a subfield which is finitely generated over $\QQ$, with algebraic closure $\bar{L}$, and let $X$ be a smooth and projective variety defined over $L$. 
	We consider on the one hand the \'etale cohomology $H^{\bullet}_{\text{\'et}}(X_{\bar{L}},\QQ_{\ell})$ for some prime number $\ell$, which is equipped with a continuous action of the absolute Galois group $\mathrm{Gal}(\bar{L}/L)$ of $L$; we let $\mathcal{G}_{\ell}(X) \subset \GL(H^{\bullet}_{\text{\'et}}(X_{\bar{L}},\QQ_{\ell}))$ denote the Zariski closure of the image of the Galois group via this representation. 
	On the other hand, we consider the singular cohomology $H^{\bullet}(X_{\mathbb{C}},\QQ)$, equipped with its natural Hodge structure; we denote by $\MT(X)\subset \GL(H^{\bullet}(X_{\mathbb{C}},\QQ))$ the Mumford--Tate group of this Hodge structure.
	Artin's comparison theorem provides an isomorphism of $\QQ_{\ell}$-algebras:
	\[
	H^{\bullet}(X_{\mathbb{C}},\QQ) \otimes_{\QQ} {\QQ_{\ell}} \xrightarrow{\ \sim \ } H^{\bullet}_{{\text{\'et}}}(X_{\bar{L}},\QQ_{\ell}).
	\]
	The Mumford--Tate conjecture is the following statement (see for instance \cite[\textbf{2.1}]{moonen2017}).
	\begin{conjecture}[Mumford--Tate conjecture]
		Artin's comparison theorem identifies $\MT(X)\otimes \QQ_{\ell}$ with the connected component of the identity of $\mathcal{G}_{\ell}(X)$. 
	\end{conjecture}
	If the Mumford--Tate conjecture holds for $X$, then the Hodge conjecture for $X_{\CC}$ and its powers is equivalent to the strong Tate conjecture for $X$ and its powers; see \cite[Proposition 2.3.2]{moonen17}.
	
	In this setting, we obtain the following result, a special case of which is Theorem \ref{thm:TateConjectureOG6resolutions} from the introduction.
	
	\begin{corollary}\label{cor:TateConjecture}
		Let the notation be as in Corollary \ref{cor:applicationHC}. Assume that $B_K$, $\widetilde{K}$, $Z_K$, and $S_K$ are all defined over a subfield $L\subset \CC$ which is finitely generated over $\QQ$. Then, for any prime number $\ell$, the strong Tate conjecture holds for any variety $X$ over $L$ isomorphic to a product of powers of $B_K$, $\widetilde{K}$, $Z_K$ and $S_K$.
	\end{corollary}
	\begin{proof}
		By Corollary \ref{cor:applicationHC}, the Hodge conjecture holds for $X_{\CC}$ and all of its powers. It will thus be sufficient to prove that the Mumford--Tate conjecture holds for $X$. By Commelin \cite[Theorem 10.3]{commelin2019}, it is enough to show that the Mumford--Tate conjecture holds for each of the varieties $B_K$, $\widetilde{K}$, $Z_K$ and $S_K$, since we know from Theorem \ref{thm:motivated} that these varieties all have abelian motive. The Mumford--Tate conjecture holds for $S_K$, as it holds for any K3 surface by \cite{Tan91, Andre1996}. It holds for $\widetilde{K}$ and~$Z_K$, because it has been proven for any hyper-K\"ahler variety of known deformation type in \cite{floccari2019, soldatenkov19, FFZ}. 
		As for $B_K$, while the Mumford--Tate conjecture is still open for abelian varieties of dimension $4$, it holds for any abelian fourfold of Weil type. Indeed, following \cite[\textbf{2.4.7}]{moonen17}, the conjecture holds for abelian varieties of dimension $\leq 3$ (and products of them, by \cite[Theorem 1.2]{commelin2019}) and for simple abelian fourfolds $A$ with $\End_{\QQ}(A)\neq \QQ$.
	\end{proof}

	\bibliographystyle{smfplain}
	\bibliography{bibliography}{}
	
\end{document}